\documentclass[11pt]{article}
\usepackage{algorithm}
\usepackage{algorithmic}

\usepackage{epsfig}
\input{amssym}
\input{epsf}
\newtheorem{prethm}{{\bf Theorem}}

\newenvironment{theorem}{\begin{prethm}{\hspace{-0.5
em}{\bf.}}}{\end{prethm}}
\newtheorem{precon}{{\bf Conjecture}}

\newenvironment{con}{\begin{precon}{\hspace{-0.5
em}{\bf.}}}{\end{precon}}
\newtheorem{prelem}{{\bf Lemma}}

\newtheorem{preobs}{{\bf Observation}}

\newtheorem{preex}{{\bf Example}}

\newtheorem{presol}{{\bf Solution}}

\newtheorem{precor}{{\bf Corollary}}

\newenvironment{cor}{\begin{precor}{\hspace{-0.5
em}{\bf.}}}{\end{precor}}

\newtheorem{preproof}{{\bf Proof.}}

\newenvironment{proof}[1]{\begin{preproof}{\rm
#1}\hfill{$\rule{2mm}{2mm}$}}{\end{preproof}}
\newtheorem{prerem}{{\bf Remark}}

\newenvironment{remark}{\begin{prerem}{\hspace{-.9
em}{\bf.}}}{\end{prerem}}
\newcommand{\FTC}[1]{Introduction}
\input{amssym}

\date{}
\title{{\Large\bf Injective chromatic number of outerplanar graphs}}

{\small
\author{
{\sc Mahsa Mozafari-Nia} %\footnote{alibehtoei@math.iut.ac.ir} 
\ and
{\sc Behnaz Omoomi}%\footnote{bomoomi@cc.iut.ac.ir}
\\ [1mm]
{\small \it  Department of Mathematical Sciences}\\
{\small \it  Isfahan University of Technology} \\
{\small \it 84156-83111, \ Isfahan, Iran}}

\begin{document}

\maketitle

\begin{abstract}
An injective coloring of a graph is a vertex coloring where two vertices with common neighbor receive distinct colors. The minimum integer $k$ that $G$ has a $k-$injective coloring is called injective chromatic number of $G$ and denoted by $\chi_i(G)$. In this paper, the injective chromatic number of outerplanar graphs with maximum degree $\Delta$ and girth $g$ is studied. It is shown that for every outerplanar graph, $\chi_i(G)\leq \Delta+2$, and this bound is tight. Then, it is proved that for outerplanar graphs with  $\Delta=3$,  $\chi_i(G)\leq \Delta+1$ and the bound is tight for outerplanar graphs of girth three and $4$. Finally, it is proved that, the injective chromatic number of $2-$connected outerplanar graphs with $\Delta=3$, $g\geq 6$ and $\Delta\geq 4$, $g\geq 4$ is equal to $\Delta$.
\end{abstract}
\textbf {Keywords:} Injective coloring, Injective chromatic number, Outerplanar graph.

\section{\FTC{1}}
All graphs we have considered here are finite, connected and simple. A plane graph is a planar drawing of a planar graph in the
Euclidean plane. The vertex set, edge set, face set, minimum degree and maximum degree of a plane graph G, are denoted by $V(G)$, $E(G)$, $F(G)$, $\delta(G)$ and $\Delta(G)$, respectively. A vertex of degree $k$ is called a $k-$vertex. For vertex $v\in V(G)$, $N_G(v)$ is the set of neighbors of $v$ in $G$.
 The girth of a graph G, $g(G)$, is the length of a shortest cycle in $G$. If there is no confusion, we delete $G$ in the notations. A face $f \in F (G)$ is denoted by its boundary walk $f=[v_1v_2\ldots v_k]$, where $v_1,v_2,\ldots,v_k$ are its vertices in the clockwise order. Also, the vertices $v_1$ and $v_k$ as end vertices of $f$ are denoted by ${v_{_L}}_{_{_f}}$ and ${v_{_R}}_{_{_f}}$, respectively. An outerplanar graph is a graph with a planar drawing for which all vertices belong to the outer face of the drawing. It is known that a graph $G$ is an outerplanar graph if and only if $G$ has no subdivision of complete graph $K_4$ and complete bipartite graph $K_{2,3}$. A path $P:v_1,v_2,\ldots,v_k$ is called a simple path in $G$ if $v_2,\ldots,v_{k-1}$ are all $2-$vertices in $G$. The length of a path is the number of its edges. We say that a face $f=[v_1v_2\ldots v_k]$ is an end face of an outerplane graph $G$, if $P:v_1,v_2,\ldots,v_k$ is a simple path in $G$. An end block in graph $G$ is a maximal $2-$connected subgraph of $G$ contains an unique cut vertex of $G$.

A proper $k-$coloring of a graph $G$ is a mapping from $V(G)$ to the set of colors $\{1,2,\ldots,k\}$ such that any two adjacent
vertices have different colors. The chromatic number, $\chi(G)$, is a minimum integer $k$ that $G$ has a proper $k-$coloring.
A coloring $c$ of $G$ is called an {\it{injective coloring}} if for every two vertices $u$ and $v$ which have common neighbor, $c(u)\neq c(v)$. That means, the restriction of $c$ to the neighborhood of any vertex is an injective function. The {\it{injective chromatic number}}, $\chi_i(G)$, is the least integer $k$
such that $G$ has an injective $k-$coloring. Note that an injective coloring is not necessarily a proper coloring. In fact, $\chi_i(G)=\chi(G^{(2)})$, where $V(G^{(2)})=V(G)$ and $uv \in E(G^{(2)})$ if and only if $u$ and $v$ have a common neighbor in $G$.
 The square of graph $G$, denoted by $G^2$, is a graph with vertex set $V(G)$, where two vertices are adjacent in $G^2$ if and only if they are at distance at most two in $G$. Since $G^{(2)}$ is a subgraph of $G^2$, obviously, $\chi_i(G)\leq \chi(G^2)$. The concept of injective coloring is introduced by Hahn et al.  in $2002$ \cite{Ontheinjectivechromaticnumberofgraphs}.
It is clear that for every graph $G$, $\chi_i(G)\geq \Delta$. In general, in \cite{Ontheinjectivechromaticnumberofgraphs} Hahn et al. proved that $\Delta\leq \chi_i(G)\leq \Delta^2-\Delta+1$.\\
 In \cite{Graphswithgivendiameterandacoloringproblem}, Wegner expressed the below conjecture for the chromatic number of the square of planar graphs.
\begin{con} {\em{\cite{Graphswithgivendiameterandacoloringproblem}}}\label{c2}
If $G$ is a planar graph with maximum degree $\Delta$, then
\begin{itemize}
\item{For $\Delta=3$, $\chi(G^2)\leq \Delta+2$.}
\item{For $4\leq \Delta\leq 7$, $\chi(G^2)\leq \Delta+5$.}
\item{For $\Delta\geq 8$, $\chi(G^2)\leq \lfloor\frac{3\Delta}{2}\rfloor +1$.}
\end{itemize}
\end{con}
Since $\chi_i(G)\leq \chi(G^2)$, Lu\v{z}ar and \v{S}krekovski in \cite{Counterexamplestoaconjectureoninjectivecoloring} proposed the following conjecture for the injective chromatic number of planar graphs. 
\begin{con} {\em{\cite{Counterexamplestoaconjectureoninjectivecoloring}}}\label{c3}
If $G$ is a planar graph with maximum degree $\Delta$, then
\begin{itemize}
\item{For $\Delta=3$, $\chi_i(G)\leq \Delta+2$.}
\item{For $4\leq \Delta\leq 7$, $\chi_i(G)\leq \Delta+5$.}
\item{For $\Delta\geq 8$, $\chi_i(G)\leq \lfloor\frac{3\Delta}{2}\rfloor +1$.}
\end{itemize}
\end{con}
The injective coloring of planar graphs with respect to its girth and maximum degree is studied in \cite{ColoringsofplanegraphsAsurvey,INJECTIVECOLORINGOFPLANARGRAPHSWITHGIRTH6,Injectivecoloringofplanargraphs,INJECTIVECOLORINGOFPLANARGRAPHSWITHGIRTH7,njectivecoloringofplanargraphswithgirth6,Injectivecoloringofplanegraphswithgirth5,Noteonplanargraphswithlargestinjectivechromaticnumbers,Injectivecoloringsofplanargraphswithfewcolors}.
In \cite{COLORINGTHESQUAREOFANOUTERPLANARGRAPH}, Lih and Wang proved upper bound $\Delta+2$ for the chromatic number of square of outerplanar graphs.
\begin{theorem} {\em {\cite{COLORINGTHESQUAREOFANOUTERPLANARGRAPH}}}\label{t1}
If $G$ is an outerplanar graph, then $\chi(G^2)\leq \Delta+2$.
\end{theorem}
Since $\chi_i(G)\leq \chi(G^2)$, Conjecture~\ref{c3} is true for outerplanar graphs.
\begin{cor}\label{c1}
If $G$ is an outerplanar graph, then $\chi_i(G)\leq \Delta+2$.
\end{cor}
In Figure~\ref{ff}, an outerplanar graph with $\Delta=4$, $g=3$ and $\chi_i(G)=\Delta+2=6$ is shown. Therefore, the given bound in Corollary~\ref{c1} is tight.
\begin{figure}[ht]
\centering
\includegraphics[scale=.7]{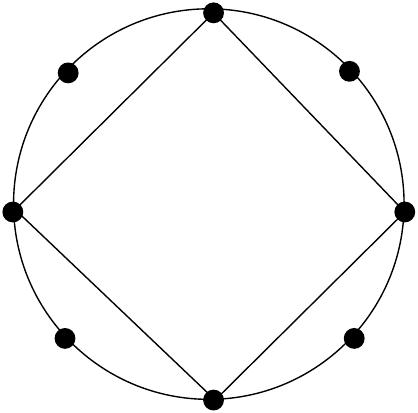} 
\caption{An outerplanar graph with $\Delta=4$, $g=3$ and $\chi_i=6$.}\label{ff}
\end{figure}

%\newpage
In this paper, we study the injective chromatic number of outerplanar graphs. The main results of Section~$2$ are as follows. \\
If $G$ is an outerplanar graph with maximum degree $\Delta$ and girth $g$, then
\begin{itemize}
\item{For $\Delta=3$ and $g\geq 3$, $\chi_i(G)\leq \Delta+1=4$}. (Theorem~\ref{t15})
\item{For $\Delta= 3$ and $g\geq 5$, with no face of degree $k$, $k\equiv 2$ \rm{(}mod $4$\rm{)}, $\chi_i(G)=\Delta$.} (Theorem~\ref{t16})
\item{For $\Delta=3$ and $g\geq 6$, $\chi_i(G)=\Delta$.} (Theorem~\ref{t20})
\item{For $\Delta\geq 4$ and $g\geq 4$, $\chi_i(G)=\Delta$.} (Theorems~\ref{t7} and \ref{t14})
\end{itemize}
\section{Main Results}
First, we prove a tight bound for the injective chromatic number of outerplanar graphs with $\Delta=3$. Note that if $\Delta=2$, then $G$ is an union of paths and cycles, which obviously \linebreak{$\chi_i(G)\leq 3=\Delta+1$}. Moreover, if $G$ is a path or an even cycle, then $\chi_i(G)=2$; and if $G$ is an odd cycle, then $\chi_i(G)=3$ \cite{Ontheinjectivechromaticnumberofgraphs}.
\begin{theorem}\label{t15}
If $G$ is an outerplanar graph with  $\Delta=3$ and $g\geq 3$, then $G$ has a $4-$injective coloring such that in every simple path of lenght three, at most three colors appear. Moreover, the bound is tight.
\end{theorem}
\begin{proof}
{We prove the theorem by the induction on $|V(G)|$. In Figure~\ref{fff}, all outerplanar graphs with $\Delta=3$  and $g\geq 3$ of order $4$ and $5$ with an injective coloring with desired property are shown. Obviously, in the left side graph, $\chi_i(G)=4$. Hence, bound $\Delta+1$ is tight.
\begin{figure}[hbtp]
\centering
\includegraphics[scale=.8]{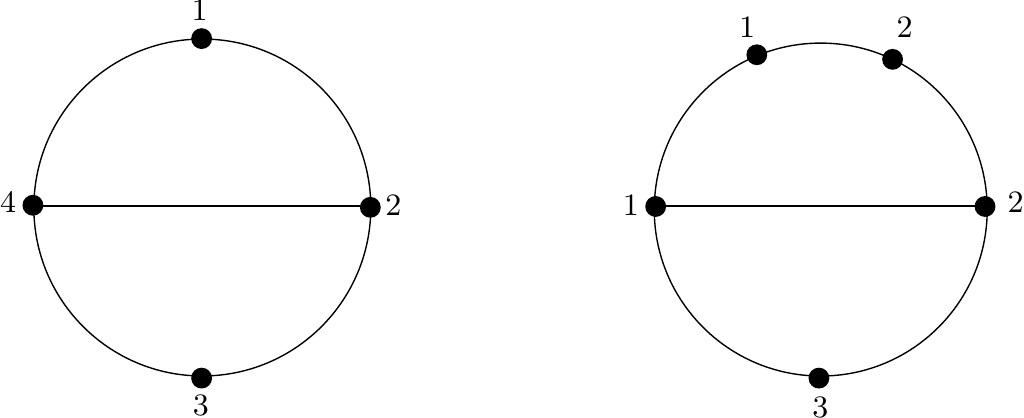}
\caption{Outerplanar graphs with $\Delta=3$ and $g\geq 3$ of order $4$, $5$.}\label{fff}
\end{figure}

Now suppose that $G$ is an outerplane graph with $\Delta=3$, $g\geq 3$ and the statement is true for all outerplanar graphs with $\Delta=3$ and 
$g\geq 3$ of order less than $|V(G)|$. The following two cases can be caused.\\
If an end block of $G$ is an edge, say $uv$, where $\deg(u)=1$, then we consider the maximal simple path $P:(v_1=u), (v_2=v),v_3,\ldots, v_k$ in $G$. Since $P$ is a maximal simple path and $\Delta(G)=3$, we have $\deg(v_k)=3$. Suppose that $N(v_k)=\{w_1,w_2,v_{k-1}\}$ and $c$ is a $4-$injective coloring of $G\setminus \{v_1,v_2,\ldots,v_{k-1}\}$ with colors $\{\alpha,\beta,\gamma,\lambda\}$ such that every simple path of length three has at most three colors. Note that $w_1$ and $w_2$ have a common neighbor $v_k$ therefore, $c(w_1)\neq c(w_2)$. In this case, we assign to the ordered vertices $v_{k-1},v_{k-2},\ldots,v_2,v_1$ of path $P$ the ordered string $(ssttsstt\ldots)$ where, $s\in \{\alpha,\beta,\gamma,\lambda\}\setminus\{c(v_k),c(w_1),c(w_2)\}$ and $t=c(v_k)$.

If the minimum degree of every end block of $G$ is at least two in $G$, then we consider an end face $f=[v_iv_{i+1} \ldots v_j]$ in an end block $B$ of $G$ in clockwise order, where $v_1$ is the vertex cut of $G$ belongs to $B$. Let $H$ be the induced subgraph of $G$ on $2-$vertices of $f$. By the induction hypothesis $G\setminus H$ has a $4-$injective coloring $c$ with colors $\{\alpha,\beta,\gamma,\lambda\}$, such that every simple path of length three has at most three colors. Hence, in $G\setminus H$ at most three colors are used for vertices $v_{i-1},v_i,v_j,v_{j+1}$.  Now we extend $c$ to an injective coloring of $G$ with the desired property.

If $c(v_i)=c(v_j)$, then we assign to the ordered vertices $v_{i+1},v_{i+2},\ldots,v_{j-1}$ the ordered string $(ssttsstt\ldots)$ where,
 $s\in\{\alpha,\beta,\gamma,\lambda\}\setminus\{c(v_{i-1}),c(v_i)=c(v_j),c(v_{j+1})\}$ and $t\in\{\alpha,\beta,\gamma,\lambda\}\setminus\{c(v_i)=c(v_j),c(v_{j+1}),s\}$. If $v_i=v_j$ and $j-i-1 \equiv 1,2$ {\rm(}mod $4${\rm)}, then change the color of $v_{j-1}$ to $t^\prime\in \{\alpha,\beta,\gamma,\lambda\}\setminus \{c(v_{i-1})=c(v_{j+1}),s,t\}$.
 
If $c(v_i)\neq c(v_j)$, then we assign to the ordered vertices $v_{i+1},v_{i+2},\ldots,v_{j-1}$ the ordered string $(ssttsstt\ldots)$ where, $s\in \{\alpha,\beta,\gamma,\lambda\}\setminus \{c(v_{i-1}),c(v_i),c(v_j),c(v_{j+1})\}$.
 If $j-i-1 \equiv 1,2$ {\rm(}mod $4${\rm)}, then $t\in \{\alpha,\beta,\gamma,\lambda\}\setminus \{c(v_j),s\}$.
 If $j-i-1 \equiv 0,3$ {\rm(}mod $4${\rm)}, then  $t\in \{\alpha,\beta,\gamma,\lambda\}\setminus \{c(v_i),c(v_{j+1}),s\}$. In the case $j-i-1\equiv 0$ {\rm(}mod $4${\rm)}, if $t=c(v_j)$, then change the color $v_{j-2}$ to $t^\prime\in \{\alpha,\beta,\gamma,\lambda\}\setminus \{c(v_j)=t,s\}$. Note that, since by the induction hypothesis $|\{c(v_{i-1}),c(v_i),c(v_j),c(v_{j+1})\}|\leq 3$, in each cases the colors $s$ and $t$ exist. It can be easily seen that the given coloring is a $4-$injective coloring for $G$ such that every simple path of length three in $G$ has at most three colors as well. }
\end{proof} 
Graph $G$ in Figure~\ref{f2} is an outerplanar graph of girth $4$ with maximum degree three and injective chromatic number $4$. Since each pair of set $\{u,v,w\}$ have a common neighbor, in every injective coloring of $G$, they must have three different colors. In the similar way, we need three different colors for the vetrices $\{x,y,z\}$. Without loss of generality, color the vertices $u,v,w$ with color $\alpha$, $\beta$ and $\gamma$, respectively. Now by devoting any permutation of these colors to vertices $x,y$ and $z$, it can be checked that in each case we need a new color for the other vertices. Therefore, bound $\Delta+1$ in Theorem~\ref{t15} is tight for outerplanar graphs with $\Delta=3$, $g=4$ and $g=3$ (see also Figure~\ref{fff}).
\begin{figure}[hbtp]
\centering
\includegraphics[scale=.6]{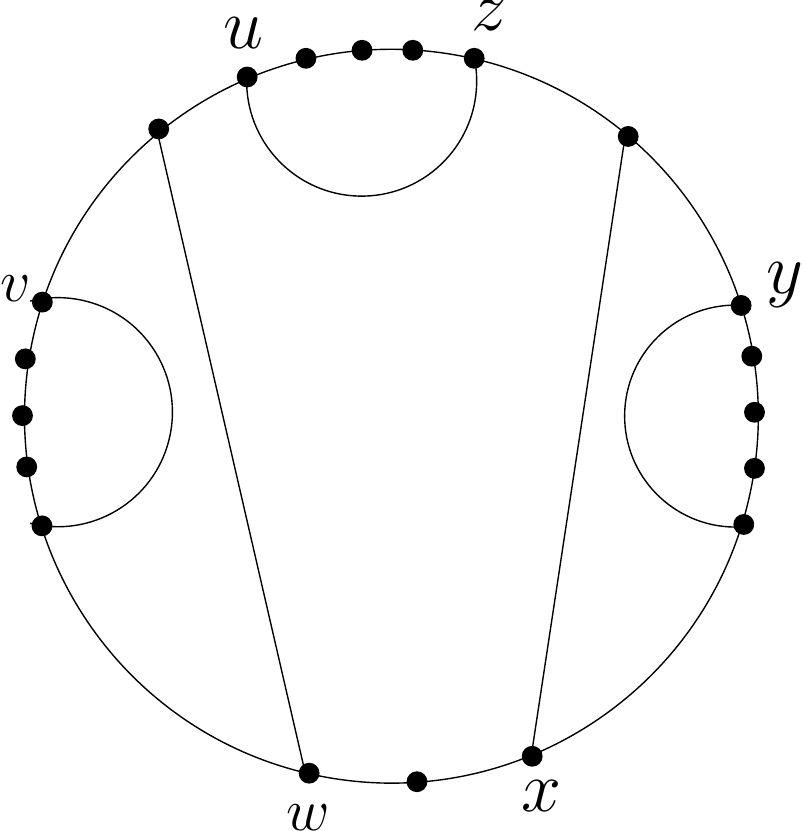}
\caption{An outerplanar graph with $\Delta=3$, $g=4$ and $\chi_i=4$.}\label{f2}
\end{figure}
%\begin{figure}[h]\hspace{5cm}
%\vspace*{0.3cm} \vspace*{0.1cm}
%\includegraphics[scale=0.32]{7}
%\caption{An outerplanar graph with $\Delta=3$, $g=4$ and $\chi_i=4$.}\label{f2}
%\end{figure}
\newpage
In the next theorems, we improve bound $\Delta+1$ to $\Delta$ for outerplanar graphs with $\Delta=3$ of girth greater than $4$. 
\begin{theorem}\label{t16}
If $G$ is a $2-$connected outerplanar graph with $\Delta=3$, $g\geq 5$ and no face of degree $k$, where $k\equiv 2$ {\rm(}mod $4${\rm)}, then $G$ has a $3-$injective coloring such that in every simple path of lenght three, exactly three colors appear.
\end{theorem}
\begin{proof}
{
Since $\chi_i(G)\geq \Delta(G)$, it is enough to show that $\chi_i(G)\leq \Delta(G)$. We prove it by the induction on $|V(G)|$. In Figure~\ref{f8}, the $2-$connected outerplanar graph with $\Delta=3$ and $g\geq 5$ of order $8$ with an injective coloring with desired property is shown.
\begin{figure}[hbtp]
\centering
\includegraphics[scale=1.1]{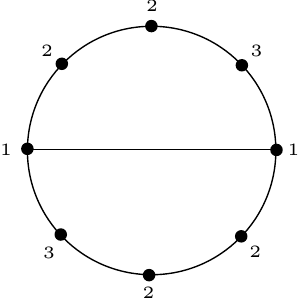}
\caption{Outerplanar graph with $\Delta=3$ and $g\geq 5$ of order $8$.}\label{f8}
\end{figure}\\
Now suppose that $G$ is a $2-$connected outerplane graph with $\Delta=3$, $g\geq 5$ and no face of degree $k$, where $k\equiv 2$ {\rm(}mod $4${\rm)} and the statement is true for all such $2-$connected outerplanar graphs of order less than $|V(G)|$. 

Let $f=[v_iv_{i+1} \ldots v_j]$ be an end face of $G$ in clockwise order and $H$ be the induced subgraph of $G$ on $2-$vertices of $f$. If $\Delta(G\setminus H)=3$, then by the induction hypothesis $G\setminus H$ has a $3-$injective coloring $c$ with colors $\{\alpha,\beta,\gamma\}$, such that every simple path of length three has exactly three colors.
% Hence, in $G\setminus H$ exactly $3$ colors used for vertices $v_{i-1},v_i,v_j,v_{j+1}$. 

 If $\Delta(G\setminus H)=2$, then we color the vertices of $G\setminus H$ as follows. If $G\setminus H=C_k$, where $k>5$ and $k\equiv 0,1$ {\rm(}mod $3${\rm)}, then color the ordered vertices $v_{i-1},v_i,v_j,v_{j+1},\ldots,v_{i-2}$ with the ordered string $(\alpha\beta\gamma\alpha\beta\gamma\ldots)$. If $k>5$ and $k\equiv 2$ {\rm(}mod $3${\rm)}, then color the ordered vertices $v_{i-1},v_i,v_j,v_{j+1},\ldots,v_{i-5}$ with the ordered string $(\alpha\beta\gamma\alpha\beta\gamma\ldots)$. Then color the vertices $v_{i-4}$, $v_{i-3}$ and $v_{i-2}$ with colors $\beta$, $\gamma$ and $\alpha$, respectively. One can check that every simple path of length three in $G\setminus H$ has exactly three colors. If $G\setminus H=C_5$, then since $|V(G)|>8$, $f=[v_iv_{i+1} \ldots v_j]$ is a cycle of length at least $6$. In this case, we consider the end face $f^\prime=[v_jv_{j+1}\ldots v_i]$ and follow the above proof when $H$ is induced subgraph of $G$ on $2-$vertices of $f^\prime$.
In the following, we extend injective coloring $c$ of $G\setminus H$ to an injective coloring of $G$ with the desired property.

If $c(v_i)=c(v_j)$, then we assign to the ordered vertices $v_{i+1},v_{i+2},\ldots,v_{j-1}$ the ordered string $(s_1s_2s_3s_4s_1s_2s_3s_4\ldots)$, where $s_1=c(v_{j+1})$. Since, $G$ has no face of degree $k$ where $k\equiv 2$ {\rm(}mod $4${\rm)}, we have following cases. If $j-i-1\equiv 1$ {\rm(}mod $4${\rm)}, then let $s_2=c(v_{i-1})$, $s_3=s_4=c(v_i)=c(v_j)$ and change the color of vertices $v_{j-2}$ and $v_{j-1}$ to $c(v_{j+1})$ and $c(v_{i-1})$, respectively. If $j-i-1\equiv 2$ {\rm(}mod $4${\rm)}, then let $s_2=s_1$, $s_3=c(v_{i-1})$ and $s_4=c(v_i)=c(v_j)$ and change the color of $v_{j-1}$ to $c(v_{i-1})$. If $j-i-1\equiv 3$ {\rm(}mod $4${\rm)}, then let $s_2=s_1$, $s_3=c(v_{i-1})$ and $s_4=c(v_i)=c(v_j)$.

If $c(v_i)\neq c(v_j)$ and $c(v_{i-1})=c(v_{j+1})$, then we assign to the ordered vertices $v_{i+1},v_{i+2},\linebreak{\ldots,v_{j-1}}$ the ordered string $(s_1s_2s_3s_4s_1s_2s_3s_4\ldots)$, where $s_1=c(v_i)$.
 If $j-i-1 \equiv 1,2$ {\rm(}mod $4${\rm)}, then $s_2=c(v_j)$ and $s_3=s_4=c(v_{i-1})=c(v_{j+1})$. In  the case  $j-i-1 \equiv 1$ {\rm(}mod $4${\rm)}, we change the color of vertices $v_{j-2}$ and $v_{j-1}$ to $c(v_i)$ and $c(v_j)$, respectively.
 If $j-i-1 \equiv 3$ {\rm(}mod $4${\rm)}, then let $s_2=c(v_{i-1})=c(v_{j+1})$ and $s_3=s_4=c(v_j)$.
 
 If $c(v_i)\neq c(v_j)$ and $c(v_{i-1})=c(v_i)$, then we assign to the ordered vertices $v_{i+1},v_{i+2},\ldots,\linebreak{v_{j-1}}$ the ordered string $(s_1s_2s_3s_4s_1s_2s_3s_4\ldots)$, where $s_1=c(v_{j+1})$.
 If  $j-i-1 \equiv 1$ {\rm(}mod~$4${\rm)}, then let $s_2=c(v_j)$, $s_3=c(v_i)=c(v_{i-1})$, $s_4=s_1$ and change the color of vertex $v_{j-1}$ to $c(v_j)$.
 If $j-i-1 \equiv 2$ {\rm(}mod $4${\rm)}, then let $s_2=c(v_j)$ and $s_3=s_4=c(v_{i-1})=c(v_i)$.
 If $j-i-1 \equiv 3$ {\rm(}mod $4${\rm)}, then we assign to the ordered vertices $v_{j-1},v_{j-2},\ldots,v_{i+1}$ the ordered string $(s_1s_2s_3s_4s_1s_2s_3s_4\ldots)$, where $s_1=c(v_j)$, $s_2=c(v_{j+1})$, $s_3=s_4=c(v_i)=c(v_{i-1})$ and change the colors of $v_{i+1}$ to $c(v_{j+1})$.
 
  If $c(v_i)\neq c(v_j)$ and $c(v_{j})=c(v_{j+1})$, then we assign to the ordered vertices $v_{j-1},v_{j-2},\ldots,\linebreak{v_{i+1}}$ the ordered string $(s_1s_2s_3s_4s_1s_2s_3s_4\ldots)$, where $s_1=c(v_{i-1})$.
 If  $j-i-1 \equiv 1$ {\rm(}mod~$4${\rm)}, then $s_2=c(v_i)$, $s_3=c(v_{j+1})=c(v_j)$, $s_4=s_1$ and change the color of $v_{i+1}$ to $c(v_{i})$.
 If $j-i-1 \equiv 2$ {\rm(}mod $4${\rm)}, then $s_2=c(v_i)$ and $s_3=s_4=c(v_{j+1})$.
 If $j-i-1 \equiv 3$ {\rm(}mod $4${\rm)}, then we assign to the ordered vertices  $v_{i+1},v_{i+2},\ldots,v_{j-1}$ the ordered string $(s_1s_2s_3s_4s_1s_2s_3s_4\ldots)$ where, $s_1=c(v_i)$, $s_2=c(v_{i-1})$ and $s_3=s_4=c(v_j)=c(v_{j+1})$ and change the color of $v_{j-1}$ to $c(v_{i-1})$.
 It can be seen that the given coloring is a $4-$injective coloring for $G$ such that every simple path of length three in $G$ has at most three colors.}
\end{proof}
%\begin{theorem}\label{ttt}
%If $G$ is a $2-$connected outerplanar graph with $\Delta=3$ and $g\geq 5$ with no face of degree $i$, where $i\equiv 2$ \rm{(}mod $4$\rm{)} and $i\equiv 0$  \rm{(}mod $3$\rm{)} , then $\chi_i(G)=3$
%\end{theorem}
%\begin{proof}
%{For each face $f$ of $G$ of length $i$, where $i\equiv 2$ {\rm(}mod $4${\rm)}, contract $2-$vertices of $f$ such that $G$ be an outerplnar graph with $\Delta=3$, $g\geq 5$ and no face of degree $i$, where $i\equiv 2$ {\rm(}mod $4${\rm)}. Now by Theorem~\ref{t16}, the obtained graphs has a $3-$injective coloring such that in every simple path of lenght $3$, exactly $3$ colors appear. Now extand all $2-$ vertices and color them with one of their available colors such that the obtained coloring be a $3-$injective coloring of $G$.}
%\end{proof}
%\begin{theorem}\label{t16}
In Theorem~\ref{t20}, we improve bound $\Delta+1$ in Theorem~\ref{t15} to $\Delta$ for outerplanar graph with $\Delta=3$ and $g\geq 6$. First, we need the following theorem.
\begin{theorem} {\em{\cite{Coloringtheverticesofagraphinprescribedcolors}}}\label{t2}
Let $G$ be a connected graph and $L$ be a list-assignment to the vertices, where $|L(v)|\geq \deg(v)$ for each $v\in V(G)$. If
\begin{enumerate}
\item{$|L(v)|>\deg(v)$ for some vertex $v$, or}
\item{$G$ contains a block which is neither a complete graph nor an induced odd cycle,}
\end{enumerate}
then $G$ admits a proper coloring such that the color assign to each vertex $v$ is in $L(v)$.
\end{theorem}
\begin{theorem}\label{t20}
If $G$ is an outerplanar graph with $\Delta=3$ and $g\geq 6$, then $\chi_i(G)=\Delta$.
\end{theorem}
\begin{proof}
{
Since $\chi_i(G)\geq \Delta$, it is enough to show that $\chi_i(G)\leq \Delta$.
Let $G$ be a minimal counterexample for this statement. That means $G$ is an outerplane graph with $\Delta=3$, $g\geq 6$ and $\chi_i(G)\geq \Delta+1$, such that every proper subgraph of $G$ has a $\Delta-$injective coloring. Obviously $\delta(G)\geq 2$. Now consider an end face $f=[v_iv_{i+1}\ldots v_j]$ in an end block $B$ of $G$ in clockwise order, where $v_1$ is the vertex cut of $G$ belongs to $B$. Since $\Delta=3$ and $g\geq 6$, the degree of face $f$ is at least $6$ and the degree of $v_i$ and $v_j$  are three.
If $\Delta(G\setminus H)=3$, then by the minimality  of $G$, we have 
$\chi_i(G \setminus H)\leq \Delta(G \setminus H)\leq \Delta(G)$. Also, if $G\setminus H$ is a cycle, then $\chi_i(G\setminus H)\leq 3= \Delta$.\\
Now, we extend the $\Delta-$injective coloring of $G\setminus H$ to a $\Delta-$injective coloring of $G$, which contradict our assumption. Each of the vertices $v_i$ and $v_j$ has at most $\Delta-1=2$ neighbors except $v_{i+1}$ and $v_{j-1}$, respectively.
 Hence, for each of vertices $v_{i+1}$ and $v_{j-1}$ there is at least one available color. Also, among the colored vertices in $G \setminus H$, the only forbbiden colors for vertices $v_{i+2}$ and $v_{j-2}$ are colors of the vertices $v_i$ and $v_j$, respectively. The other vertices have three available colors. Now consider induced subgraph of $G^{(2)}$ on the vertices of $H$, denoted by $G^{(2)}[H]$, and list of available colors for each vertex of $H$. The components of $G^{(2)}[H]$ are some paths or cycles satisfying the assumption of Theorem~\ref{t2}. Thus, we have a proper $\Delta-$coloring for $G^{(2)}[H]$  using the available colors which is a $\Delta-$injective coloing of $H$ as desired.}
\end{proof}

Now we are ready to determine the injective chromatic number of $2-$connected outerplanar graphs with maximum degree and girth greater than three. We prove this fact by two different methods for the cases $\Delta=4$ and $\Delta\geq 5$.
\begin{theorem}\label{t7}
If $G$ is a $2-$connected outerplanar graph with $\Delta=4$ and $g\geq 4$, then $G$ has a $4-$injective coloring $c$ such that for every adjacent vertices $v$ and $u$ of degree three with $N(v)=\{u,v_1,v_2\}$ and $N(u)=\{v,u_1,u_2\}$, $\{c(u),c(v_1),c(v_2)\}\neq \{c(v),c(u_1),c(u_2)\}$.
\end{theorem}
\begin{proof}
{Since $\chi_i(G)\geq \Delta(G)=4$, it is enough to show that $\chi_i(G)\leq 4$. We prove it by the induction on $|V(G)|$. In Figure~\ref{f7}, the $2-$connected outerplanar graph with $\Delta=4$ and $g\geq 4$ of order $8$  with an injective coloring of desired property is shown.
\begin{figure}[hbtp]
\centering
\hspace*{3cm}
\includegraphics[scale=.6]{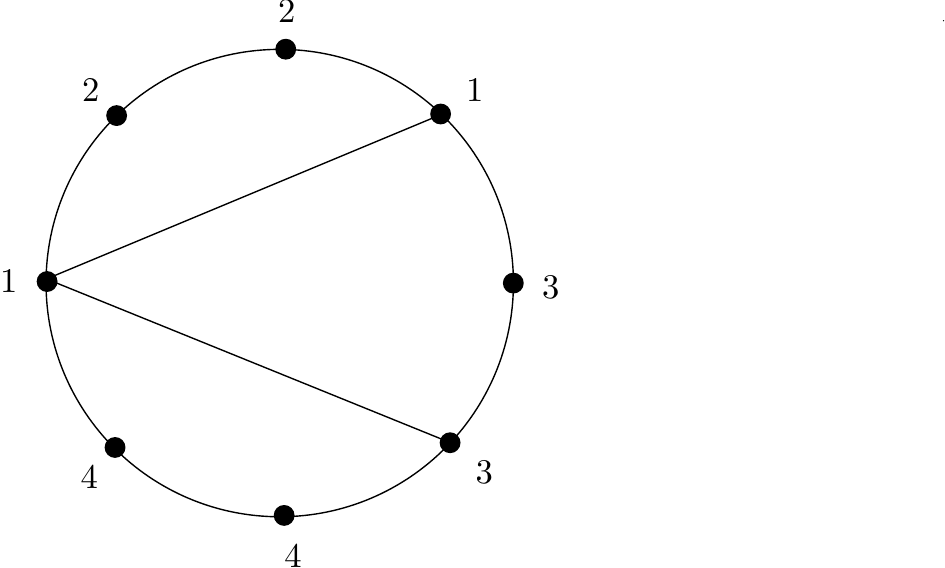}
\caption{$2-$connected outerplanar graph with $\Delta=4$ and $g\geq 4$ of order $8$.}\label{f7}
\end{figure}

Now suppose that $G$ is a $2-$connected outerplane graph with $\Delta=4$, $g\geq 4$ and the statement is true for all $2-$connected outerplanar graphs with $\Delta=4$ and $g\geq 4$ of order less than $|V(G)|$.

Let $f=[v_iv_{i+1} \ldots v_j]$ be an end face of $G$ in clockwise order.
If $\deg(v_i)=\deg(v_j)=3$, then consider induced subgraph $H$ on $2-$vertices of face $f$. Thus, $G\setminus H$ is a $2-$connected outerplane graph with $\Delta(G\setminus H)=4$ and $g(G\setminus H)\geq 4$. Hence, by the induction hypothesis, $G\setminus H$ has a $4-$injective coloring such that for every adjacent vertices $v$ and $u$ of degree three with $N(v)=\{u,v_1,v_2\}$ and $N(u)=\{v,u_1,u_2\}$, $\{c(u),c(v_1),c(v_2)\}\neq \{c(v),c(u_1),c(u_2)\}$. 
If there are exactly four colors in $\{c(v_{i-1}),c(v_i),c(v_j),c(v_{j+1})\}$, then consider graph $G^{(2)}[H]$ and list of available colors for each vertex of $H$. Graph $G^{(2)}[H]$ satisfy the assumption of Theorem~\ref{t2}. Thus, we have a $\Delta-$coloring for $G^{(2)}[H]$ which is a $\Delta-$injective coloring of $G$.
 If there are at most three colors in $\{c(v_{i-1}),c(v_i),c(v_j),c(v_{j+1})\}$, then color $v_{i+1}$ with one of its colors not in $\{c(v_{i-1}),c(v_i),c(v_j),c(v_{j+1})\}$ and color $v_{j-1}$ with one of its available colors such that $c(v_{i+1})\neq c(v_{j-1})$. Then color the other vertices of $H$ with one of their available colors similar to above.
It can be easily seen that for every adjacent vertices $v$ and $u$ of degree three with $N(v)=\{u,v_1,v_2\}$ and $N(u)=\{v,u_1,u_2\}$, $\{c(u),c(v_1),c(v_2)\}\neq \{c(v),c(u_1),c(u_2)\}$.

Now suppose that each face of $G$ has an end vertex of degree $4$. We have two following cases.\\
$(i)$ There is an end face $f$ with one end vertex of degree $4$ and the other one of degree less than $4$.\\
$(ii)$ For each end face $f$, two its end vertices are of degree $4$.\\
In the former case, suppose that $G$ has an end face $f=[v_i v_{i+1}\ldots v_j]$, where $\deg(v_i)=4$ and $\deg(v_j)=3$. Consider induced subgraph $H$ on $2-$vertices of face $f$. If $\Delta(G\setminus H)=4$, then by the induction hypothesis, $G\setminus H$ has a $4-$injective coloring such that for every adjacent vertices $v$ and $u$ of degree three with $N(v)=\{u,v_1,v_2\}$ and $N(u)=\{v,u_1,u_2\}$, $\{c(u),c(v_1),c(v_2)\}\neq \{c(v),c(u_1),c(u_2)\}$. Now we extend the $4-$injective coloring of \linebreak{$G\setminus H$} to $G$.
If $\deg(v_{j+1})=3$, then suppose that $v_s$ is the other neighbor of $v_{j+1}$ except $v_j$ and $v_{j+2}$. If there are exactly three colors in $\{c(v_i),c(v_j),c(v_{j+1}),c(v_{j+2}),c(v_s)\}$, then color vertex $v_{j-1}$ with one of its colors not in $\{c(v_i),c(v_j),c(v_{j+1}),c(v_{j+2}),c(v_s)\}$ and color the other vertices of $H$ with one of their available colors as explained in above. If $|\{c(v_i),c(v_j),c(v_{j+1}),c(v_{j+2}),c(v_s)\}|=4$ or $\deg(v_{j+1})\neq 3$, then by Theorem~\ref{t2} color the vertices of $H$ with one of their available colors such that  obtained coloring is a $4-$injective coloring of $G$.

If $\Delta(G\setminus H)=3$, then by assumption $(i)$, there is another unique end face, say $f^\prime$, with a common neighbor with $f$.
Consider induced subgraph $H$ on $2-$vertices of face $f$ and  $f^\prime$. Thus, $G\setminus H$ is a cycle and $\chi_i(G\setminus H)\leq 3$. Now each vertices of $H$ has at least two available colors. Hence, by applying Theorem~\ref{t2}, we obtain a $4-$injective coloring of $G$. Note that, since $g(G)\geq 4$, in this case there is no two adjacent vertices of degree three.\\
In the latter case, consider the induced subgraph $H$ on $2-$vertices of $f=[v_iv_{i+1}\ldots v_j]$, where $\deg(v_i)=\deg(v_j)=4$. Since $\deg(v_i)=4$, $G\setminus H$ has an end face $f^\prime$ with two ends of degree $4$. Hence, $\Delta(G\setminus H)=4$ and  by the induction hypothesis,  $G\setminus H$ has a $4-$injective coloring such that for every adjacent vertices $v$ and $u$ of degree three with $N(v)=\{u,v_1,v_2\}$ and $N(u)=\{v,u_1,u_2\}$, $\{c(u),c(v_1),c(v_2)\}\neq \{c(v),c(u_1),c(u_2)\}$. Now by Theorem~\ref{t2}, color the vertices of $H$ with their available colors such that  obtained coloring is a $4-$injective coloring of $G$. Obviously, for every adjacent vertices $v$ and $u$ of degree three with $N(v)=\{u,v_1,v_2\}$ and $N(u)=\{v,u_1,u_2\}$, $\{c(u),c(v_1),c(v_2)\}\neq \{c(v),c(u_1),c(u_2)\}$. 
}
\end{proof}

Now we consider $2-$connected outerplanar graphs with $\Delta=5$ and $g\geq 4$. First, we need to prove the following theorem on the structure of $2-$connected outerplanar graphs.
\begin{theorem}\label{t9}
If $G$ is a $2-$connected outerplanar graph, then $G$ has an end face \linebreak{$f=[v_i v_{i+1} \ldots v_j]$}, where either $\deg(v_i)< 5$ or $\deg(v_j)<5$. 
\end{theorem}
\begin{proof}
{First replace every simple path in boundary of each end face of $G$ with a path of length two and name this graph $G^\prime$. Graph $G^\prime$ is also a $2-$connected outerplane graph that each end face of $G^\prime$ is of degree three. Now, let $C:v_1v_2\ldots v_n$ be a Hamilton cycle of $G^\prime$ in clockwise order and $f=[v_i v_{i+1} v_{i+2}]$ be an end face of $G^\prime$.
%Since $\deg(v_i)=4$, there exists another adjacent vertex $v_i^\prime$ to $v_i$, $i+3\leq i^\prime< i-1$. Similarly, there exists an adjacent vetrex $v_j^\prime$ to $v_{i+2}$, where $i+4<j^\prime \leq i^\prime$.
%Otherwise, $G^\prime$ has a subdivision of $K_4$ on $3-$vertices $v_i,v_{i+2},v_{i^\prime},v_{j^\prime}$. 
%The following algorithm provide a procedeure to find a subgraph isomorphism to $F$ in $G^\prime$ where each end face in $F$ is of degree $3$.
If $G^\prime$ is a cycle, then we are done. Hence, suppose that $\Delta(G^\prime)\geq 3$ and  $f=[v_iv_{i+1}v_{i+2}]$ is an end face of $G^\prime$ where, $\deg(v_i)$ and $\deg(v_{i+2})$ are at least 5. In what follows, we present an algorithm that find an end face of $G^\prime$ such that the degree of at least one of its end vertices is less than $5$. Since by assumption $\deg(v_{i+2})\geq 5$, $v_{i+2}$ has at least two other neighbors except $v_i$, $v_{i+1}$ and $v_{i+3}$, named $v_{i^\prime}$ and $v_{j^\prime}$ where, $j^\prime< i^\prime$. 
\begin{algorithm}[H]
\caption{} \label{a2}
\begin{algorithmic}[1]
\STATE $k=0$.
\STATE $f_0=[v_i v_{i+1} v_{i+2}]$.
\STATE If $f_k$ is an end face of $G$, then do steps $4$ to $7$, respectively.
\STATE Suppose that ${v_{_L}}_{_{_{f_k}}}=v_t$ and ${v_{_R}}_{_{_{f_k}}}=v_{t+2}$. Let ${v_{j^\prime}}_{_k}$  and ${v_{i^\prime}}_{_k}$ be another neighbors of $v_{t+2}$ except $v_t$, $v_{t+1}$ and $v_{t+3}$, where $j^\prime_k<i^\prime_k$.
\STATE If there is no ${v_{i^\prime}}_{_k}$ or ${v_{j^\prime}}_{_k}$, then stop the algorithm and give the face $f_k$ as output of the algorithm.
\STATE $k=k+1$.
\STATE $f_k=[{v_{_R}}_{_{f_{k-1}}} {v_{_R}}_{_{f_{k-1}}+1} \ldots {v_{j^\prime}}_{_{k-1}}]$ and go to step three.
\STATE If $f_k$ is not an end face of $G$, then there exist an end face $f^\prime$ in $f_k$. Do steps $9$ and $10$, respectively.
\STATE $k=k+1$.
\STATE $f_k=f^\prime$ and go to step three.
\end{algorithmic}
\end{algorithm}
Note that the indices of neighbors of all vertices $v_k$, $i+3\leq k< i^\prime$, is less than $i^\prime$; otherwise there is a subdivision of $K_4$ on $G^\prime$ and it is a contradiction with the assumption that $G^\prime$ is an outerplanar graph. Since in each step of the algorithm the indices of vertices in $f_k$ are increasing, the algorithm terminates and its output, say $f=[v_s v_{s+1}v_{s+2}]$, is an end face in $G^\prime$ where $v_{s+2}$ has degree at most $4$. Therefore, by returning the contracted paths to $G^\prime$; we have an end face of $G$ that one of its ends is of degree less than $5$.}
\end{proof}
\begin{theorem}\label{t14}
If $G$ is a $2-$connected outerplanar graph with $\Delta\geq 5$ and $g \geq 4$, then $\chi_i(G)=\Delta$.
\end{theorem}
\begin{proof}
{Since $\chi_i(G)\geq \Delta(G)$, it is enough to show that $\chi_i(G)\leq \Delta(G)$. We prove it by the induction on $|V(G)|$. In Figure~\ref{f10}, the $2-$connected outerplanar graph with $\Delta\geq 5$ and $g\geq 4$ of order $10$ with a $\Delta-$injective coloring is shown.
\begin{figure}[ht]
\centering
\includegraphics[scale=.65]{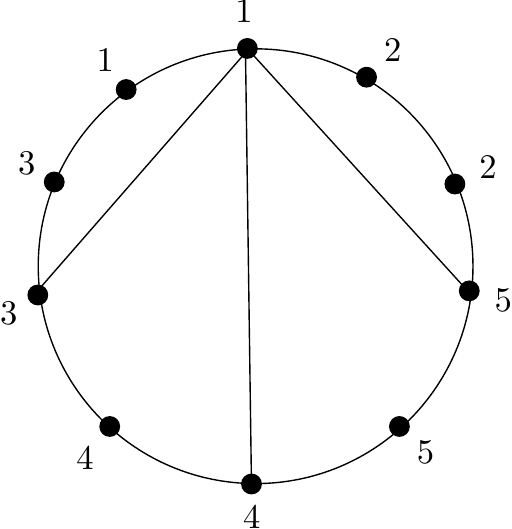}
\caption{$2-$connected outerplanar graph with $\Delta\geq 5$ and $g\geq 4$ of order $10$.}\label{f10}
\end{figure}\\
%If $G$ has an end block, say $uv$, where $\deg(u)=1$, then we consider  $G\setminus \{u\}$. If $\Delta(G\setminus \{u\})\geq 5$, then by the induction hypothesis, $\chi_i(G\setminus \{u\})=\Delta(G\setminus \{u\})\leq \Delta(G)$. Since $u$ has at most $4$ forbiddon colors, there is at least one available color for $u$. color $u$ with one of its available colors to obtain a $\Delta-$injective coloring for $G$. If $\Delta(G\setminus \{u\})=4$, then by Theorem~\ref{t16}, $G\setminus \{u\}$ has a $4-$injective coloring. Now color $u$ as before.
%Note that if $\deg(v)=5$, then $\Delta(G\setminus \{u\})\geq 4$ and if $\deg(v)\leq 4$, then there is a vertex of degree $5$ in $G\setminus \{u\}$. 
Now suppose that $G$ is a $2-$connected outerplane graph with $\Delta\geq 5$, $g\geq 4$ and the statement is true for all $2-$connected outerplanar graphs with $\Delta\geq 5$ and $g\geq 4$ of order less that $|V(G)|$.
%Let $f=[v_iv_{i+1} \ldots v_j]$ is an end face of $G$ in clockwise order and $H$ be the induced subgraph of $G$ on $2-$vertices of $f$. 
%Let $G$ be a minimal counter example of the statement. That means $G$ is an outer plane graph with $\Delta\geq 4$, $g\geq 4$ and $\chi_i(G)\geq \Delta+1$ and every proper subgraph of $G$ is $\Delta-$injective colorable. Hence, $\delta(G)\geq 2$.
%Let $f=[v_1\ldots v_k]$ be an end face of an end block of $G$ and $H$ be an induced subgraph on vertices $\{v_2,\ldots,v_{k-1}\}$. By the minimallity assumption of $G$, we have \linebreak{$\chi_i(G\setminus H)\leq \Delta(G\setminus H)\leq \Delta$}. Now we extend the $\Delta-$injective coloring of $G\setminus H$ to a $\Delta-$injective coloring of $G$, where contradict the assumption.
%To do this we consider two cases $\Delta(G)=4$ and $\Delta(G)\geq 5$.
%
By Theorem~\ref{t9}, $G$ has an end face $f$ of degree at least $4$ such that at least one of its end vertices is of degree at most $4$. Now consider the induced subgraph $H$ on $2-$vertices of end face $f$. If $\Delta(G\setminus H)\geq 5$, then by induction hypothesis, $\chi_i(G\setminus H)=\Delta(G\setminus H)\leq \Delta(G)$. If  $\Delta(G\setminus H)=4$, then by Theorem~\ref{t7}, $G\setminus H$ has a $4-$injective coloring.
Now consider the end face $f=[v_iv_{i+1}\ldots v_j]$ and suppose that $\deg(v_i)\leq \Delta$ and $\deg(v_j)\leq 4$. Since $\Delta\geq 5$, the vertices $v_{i+1}$ and $v_{j-1}$ have at least one and two available colors, respectively. The other vertices of $H$ has at least three available colors. Now consider the graph $G^{(2)}[H]$ and list of available colors for each
vertex of H. It can be easily seen that $G^{(2)}[H]$ is union of paths and isolated vertices, which satisfy the assumption of Theorem~\ref{t2}.
Hence, $G^{(2)}[H]$ can be colored by at most $\Delta$ colors and the obtained coloring is a $\Delta-$injective coloring of $H$.}
\end{proof}
\begin{remark}
We believe that, laboriously, the results of Theorems~\ref{t16}, \ref{t7} and \ref{t14} can be generalized for the outerplanar graphs  containing some cut vertices.
\end{remark}

{\bibliographystyle{plain}
%\setLTRbibitems
%\begin{singlespace}
\bibliography{./bibliography}
\end{document}